\documentclass[11pt,draft]{amsart}

\usepackage{a4,amssymb,cite,epic}

\DeclareMathOperator{\Nil}{Nil}

\DeclareMathOperator{\var}{var}

\DeclareMathOperator{\ZR}{ZR}

\newtheorem{theorem}{Theorem}[section]

\newtheorem{proposition}[theorem]{Proposition}

\newtheorem{lemma}[theorem]{Lemma}

\newtheorem{corollary}[theorem]{Corollary}

\makeatletter

\renewcommand*\subjclass[2][2000]{\def\@subjclass{#2}\@ifundefined
{subjclassname@#1}{\ClassWarning{\@classname}{Unknown edition (#1) of
Mathematics Subject Classification; using '2000'.}}{\@xp\let\@xp
\subjclassname\csname subjclassname@#1\endcsname}}

\makeatother

\begin{document}

\title[Proofs of definability of some varieties of semigroups]{Proofs of
definability of some varieties\\
and sets of varieties of semigroups}

\author[B. M. Vernikov]{B. M. Vernikov\\
\\
Communicated by L. N. Shevrin}

\address{Department of Mathematics and Mechanics, Ural State University,
Lenina 51, 620083 Ekaterinburg, Russia}

\email{bvernikov@gmail.com}

\date{}

\thanks{The work was partially supported by the Russian Foundation for Basic
Research (grants No.~09-01-12142, 10-01-00524) and the Federal Education
Agency of the Russian Federation (project No.~2.1.1/3537).}

\begin{abstract}
We show that many important varieties and sets of varieties of semigroups may
be defined by relatively simple and transparent first-order formulas in the
lattice of all semigroup varieties.
\end{abstract}



\maketitle

\section*{Introduction}
\label{intr}

A subset $A$ of a lattice $\langle L;\vee,\wedge\rangle$ is called \emph
{definable in} $L$ if there exists a first-order formula $\Phi(x)$ with one
free variable $x$ in the language of lattice operations $\vee$ and $\wedge$
which \emph{defines $A$ in} $L$. This means that, for an element $a\in L$,
the sentence $\Phi(a)$ is true if and only if $a\in A$. If $A$ consists of a
single element, we speak about definability of this element.

We denote the lattice of all semigroup varieties by \textbf{SEM}. A set of
semigroup varieties $X$ (or a single semigroup variety $\mathcal X$) is said
to be \emph{definable} if it is definable in \textbf{SEM}. In this situation
we will say that the corresponding first-order formula \emph{defines} the set
$X$ or the variety $\mathcal X$.

A number of deep results about definable varieties and sets of varieties of
semigroups have been obtained in \cite{Jezek-McKenzie-93} by Je\v{z}ek and
McKenzie\footnote{We note that the paper \cite{Jezek-McKenzie-93}, as well as
the articles \cite{Grech-09,Kisielewicz-04} mentioned in Section \ref{commut}
below have dealt with the lattice of equational theories of semigroups, that
is, the dual of \textbf{SEM} rather than the lattice \textbf{SEM} itself.
When reproducing results from \cite{Grech-09,Kisielewicz-04,
Jezek-McKenzie-93}, we adapt them to the terminology of the present note.}.
It has been conjectured there that every finitely based semigroup variety is
definable up to duality. The conjecture is confirmed in \cite
{Jezek-McKenzie-93} for locally finite finitely based varieties. However the
article \cite{Jezek-McKenzie-93} contains no explicit first-order formulas
that define any given locally finite finitely based variety. On their way to
obtain the mentioned fundamental result, Je\v{z}ek and McKenzie proved the
definability of several important sets of semigroup varieties such as the
sets of all finitely based, all locally finite, all finitely generated and
all 0-reduced semigroup varieties. But the article \cite{Jezek-McKenzie-93}
contains no explicit first-order formulas that define any of these sets of
varieties. The task of writing an explicit formula that defines the set of
all finitely based or the set of all locally finite or the set of all
finitely generated varieties seems to be extremely difficult. But on the
other hand, many traditionally considered sets of semigroup varieties
(including the set of all 0-reduced varieties) and many important individual
varieties (for instance, an arbitrary Abelian periodic group variety) can be
defined by relatively simple first-order formulas. Such formulas come
naturally from the structural theory of semigroup varieties.

Here we present explicit formulas that define several well-known sets of
semigroup varieties and individual varieties. Each of these varieties and
sets of varieties has appeared multiple number of times in many articles in
semigroup theory.

We will denote the conjunction by \& rather than $\wedge$ because the latter
symbol stands for the meet in a lattice. Since the disjunction and the join
in a lattice are denoted usually by the same symbol $\vee$, we use this
symbol for the join and denote the disjunction by \textsc{or}. Evidently, the
relations $\le$, $\ge$, $<$ and $>$ in a lattice $L$ can
be expressed in terms of, say, meet operation $\wedge$ in $L$. So, we will
freely use these four relations in formulas. Let $\Phi(x)$ be a first-order
formula. For the sake of brevity, we put
$$\min\nolimits_x\bigl\{\Phi(x)\bigr\}\;\rightleftharpoons\;\Phi(x)\,\&\,
(\forall y)\,\bigl(y<x\longrightarrow\neg\Phi(y)\bigr)$$
and
$$\max\nolimits_x\bigl\{\Phi(x)\bigr\}\;\rightleftharpoons\;\Phi(x)\,\&\,
(\forall y)\,\bigl(x<y\longrightarrow\neg\Phi(y)\bigr)\ldotp$$
Clearly, the formula $\min_x\bigl\{\Phi(x)\bigr\}$ [respectively $\max_x\bigl
\{\Phi(x)\bigr\}$] defines the set of all minimal [maximal] elements of the
set
$$\{a\in L\mid\ \text{the sentence}\ \Phi(a)\ \text{is true}\}\ldotp$$

\section{Atoms and chain varieties}
\label{atoms and chain}

Many important sets of semigroup varieties admit a characterization in the
language of atoms of the lattice \textbf{SEM}. The set of all atoms of a
lattice $L$ with 0 is defined by the formula
$$\mathtt A(x)\;\rightleftharpoons\;(\exists y)\;\bigl((\forall z)\,(y\le
z)\,\&\,\min\nolimits_x\{x\ne y\}\bigr)\ldotp$$
A description of all atoms of the lattice \textbf{SEM} is well known. To list
these varieties, we need some notation.

By $\var\Sigma$ we denote the semigroup variety given by the identity system
$\Sigma$. A pair of identities $wx=xw=w$ where the letter $x$ does not occur
in the word $w$ is usually written as the symbolic identity $w=0$\footnote
{This notation is justified because a semigroup with such identities has a
zero element and all values of the word $w$ in this semigroup are equal to
zero.}. Identities of the form $w=0$ as well as varieties given by identities
of such a form are called 0-\emph{reduced}. Let us fix notation for several
semigroup varieties:
\begin{align*}
&\mathcal A_n=\var\,\{x^ny=y,\,xy=yx\}\ \text{--- the variety of Abelian
groups}\\
&\phantom{\mathcal A_n=\var\,\{x^ny=y,\,xy=yx\}\ \text{--- }}\text{whose
exponent divides}\ n,\\
&\mathcal{SL}=\var\,\{x^2=x,\,xy=yx\}\ \text{--- the variety of
semilattices},\\
&\mathcal{LZ}=\var\,\{xy=x\}\ \text{--- the variety of left zero semigroups},
\\
&\mathcal{RZ}=\var\,\{xy=y\}\ \text{--- the variety of right zero
semigroups},\\
&\mathcal{ZM}=\var\,\{xy=0\}\ \text{--- the variety of null semigroups}\ldotp
\end{align*}

The following lemma is well known (see the surveys \cite{Evans-71,
Shevrin-Vernikov-Volkov-09}, for instance).

\begin{lemma}
\label{atoms}
The varieties $\mathcal A_p$ \textup(where $p$ is a prime number\textup),
$\mathcal{SL}$, $\mathcal{LZ}$, $\mathcal{RZ}$, $\mathcal{ZM}$ and only they
are atoms of the lattice $\mathbf{SEM}$.\qed
\end{lemma}

If $\mathcal V$ is a semigroup variety then we denote by $\overleftarrow
{\mathcal V}$ the variety \emph{dual to} $\mathcal V$, that is, the variety
consisting of semigroups anti-isomorphic to members of $\mathcal V$. The map
from the lattice \textbf{SEM} into itself given by the rule $\mathcal V
\longmapsto\overleftarrow{\mathcal V}$ is an automorphism of \textbf{SEM}.
Therefore if $\mathcal V\ne\overleftarrow{\mathcal V}$ then the variety
$\mathcal V$ is not definable. The varieties $\mathcal{LZ}$ and $\mathcal
{RZ}$ are dual to each other, whence they are not definable. We will say that
a variety $\mathcal V$ is \emph{definable up to duality} if $\mathcal V\ne
\overleftarrow{\mathcal V}$ and the set $\{\mathcal V,\overleftarrow{\mathcal
V}\}$ is definable. We are going to verify that all atoms of \textbf{SEM}
except $\mathcal{LZ}$ and $\mathcal{RZ}$ are definable, while the varieties
$\mathcal{LZ}$ and $\mathcal{RZ}$ are definable up to duality (see
Proposition \ref{definable sets of atoms} and Theorem \ref{A_n} below). To
achieve this aim, we need some additional definitions, notation and results.
Put
$$\mathtt{Neut}(x)\;\rightleftharpoons\;(\forall y,z)\;\bigl((x\vee y)\wedge
(y\vee z)\wedge(z\vee x)=(x\wedge y)\vee(y\wedge z)\vee(z\wedge x)\bigr)
\ldotp$$
An element $x$ of a lattice $L$ such that the sentence $\mathtt{Neut}(x)$ is
true is called \emph{neutral}. Neutral elements play a distinguished role in
the lattice theory (see Section III.2 in \cite{Gratzer-98}, for instance). We
denote by $\mathcal T$ the trivial semigroup variety, and by $\mathcal{SEM}$
the variety of all semigroups.

\begin{lemma}[\!\!{\mdseries\cite[Proposition 2.4]{Volkov-05}}]
\label{neutral}
The varieties $\mathcal T$, $\mathcal{SL}$, $\mathcal{ZM}$, $\mathcal{SL\vee
ZM}$, $\mathcal{SEM}$ and only they are neutral elements of the lattice
$\mathbf{SEM}$.\qed
\end{lemma}

A semigroup variety $\mathcal V$ is called \emph{chain} if the subvariety
lattice of $\mathcal V$ is a chain. Clearly, each atom of \textbf{SEM} is a
chain variety. The set of all chain varieties is definable by the formula
$$\mathtt{Ch}(x)\;\rightleftharpoons\;(\forall y,z)\,(y\le x\,\&\,z\le x
\longrightarrow y\le z\ \text{\textsc{or}}\ z\le y)\ldotp$$

We adopt the usual agreement that an adjective indicating a property shared
by all semigroups of a given variety is applied to the variety itself; the
expressions like ``completely regular variety'', ``periodic variety'',
``nil-variety'' etc.\ are understood in this sense.

Put
\begin{align*}
&\mathcal N_k=\var\,\{x^2=x_1x_2\cdots x_k=0,\,xy=yx\}\ (k\ \text{is a
natural number}),\\
&\mathcal N_\omega=\var\,\{x^2=0,\,xy=yx\},\\
&\mathcal N_3^2=\var\,\{x^2=xyz=0\},\\
&\mathcal N_3^c=\var\,\{xyz=0,\,xy=yx\}
\end{align*}
(in particular $\mathcal N_1=\mathcal T$ and $\mathcal N_2=\mathcal{ZM}$).
The following lemma is proved in \cite{Sukhanov-82}.

\begin{lemma}
\label{chain var}
The varieties $\mathcal{SL}$, $\mathcal{LZ}$, $\mathcal{RZ}$, $\mathcal N_k$,
$\mathcal N_\omega$, $\mathcal N_3^2$, $\mathcal N_3^c$ and only they are
non-group chain varieties of semigroups.\qed
\end{lemma}

The problem of a complete classification of chain group varieties seems to be
extremely difficult (see Subsection 11.6 in \cite{Shevrin-Vernikov-Volkov-09}
for additional comments on this subject). But in the Abelian case the problem
turns out to be trivial. The lattice of all Abelian periodic group varieties
is evidently isomorphic to the lattice of natural numbers ordered by
divisibility. This readily implies that non-trivial chain Abelian group
varieties are varieties $\mathcal A_{p^k}$ with prime $p$ and natural $k$,
and only they. Fig.\ \ref{chain varieties} shows the relative location of the
chain varieties mentioned above in the lattice \textbf{SEM}.

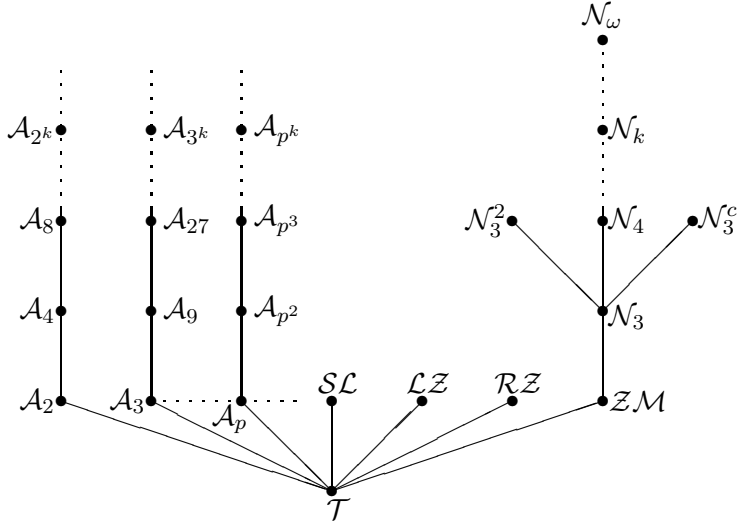
\begin{figure}[tbh]
\begin{center}
\unitlength=.8mm
\linethickness{0.4pt}
\begin{picture}(107,82)
\put(1,18){\line(0,1){32}}
\put(1,18){\line(3,-1){45}}
\put(16,18){\line(0,1){32}}
\put(16,18){\line(2,-1){30}}
\put(31,18){\line(0,1){32}}
\put(31,18){\line(1,-1){15}}
\put(46,3){\line(0,1){15}}
\put(46,3){\line(1,1){15}}
\put(46,3){\line(2,1){30}}
\put(46,3){\line(3,1){45}}
\put(76,48){\line(1,-1){15}}
\put(91,18){\line(0,1){32}}
\put(91,33){\line(1,1){15}}
\dashline{.5}(1,50)(1,73)
\dashline{.5}(16,50)(16,73)
\dashline{.5}(18,18)(40,18)
\dashline{.5}(31,50)(31,73)
\dashline{.5}(91,50)(91,76)
\put(1,18){\circle*{1.62}}
\put(1,33){\circle*{1.62}}
\put(1,48){\circle*{1.62}}
\put(1,63){\circle*{1.62}}
\put(16,18){\circle*{1.62}}
\put(16,33){\circle*{1.62}}
\put(16,48){\circle*{1.62}}
\put(16,63){\circle*{1.62}}
\put(31,18){\circle*{1.62}}
\put(31,33){\circle*{1.62}}
\put(31,48){\circle*{1.62}}
\put(31,63){\circle*{1.62}}
\put(46,3){\circle*{1.62}}
\put(46,18){\circle*{1.62}}
\put(61,18){\circle*{1.62}}
\put(76,18){\circle*{1.62}}
\put(76,48){\circle*{1.62}}
\put(91,18){\circle*{1.62}}
\put(91,33){\circle*{1.62}}
\put(91,48){\circle*{1.62}}
\put(91,63){\circle*{1.62}}
\put(91,78){\circle*{1.62}}
\put(106,48){\circle*{1.62}}
\put(0,18){\makebox(0,0)[rc]{$\mathcal A_2$}}
\put(0,33){\makebox(0,0)[rc]{$\mathcal A_4$}}
\put(0,48){\makebox(0,0)[rc]{$\mathcal A_8$}}
\put(0,63){\makebox(0,0)[rc]{$\mathcal A_{2^k}$}}
\put(12,18){\makebox(0,0)[cc]{$\mathcal A_3$}}
\put(18,33){\makebox(0,0)[lc]{$\mathcal A_9$}}
\put(18,48){\makebox(0,0)[lc]{$\mathcal A_{27}$}}
\put(18,63){\makebox(0,0)[lc]{$\mathcal A_{3^k}$}}
\put(29,15){\makebox(0,0)[cc]{$\mathcal A_p$}}
\put(33,33){\makebox(0,0)[lc]{$\mathcal A_{p^2}$}}
\put(33,48){\makebox(0,0)[lc]{$\mathcal A_{p^3}$}}
\put(33,63){\makebox(0,0)[lc]{$\mathcal A_{p^k}$}}
\put(62,21){\makebox(0,0)[cc]{$\mathcal{LZ}$}}
\put(92,32){\makebox(0,0)[lc]{$\mathcal N_3$}}
\put(75,48){\makebox(0,0)[rc]{$\mathcal N_3^2$}}
\put(107,48){\makebox(0,0)[lc]{$\mathcal N_3^c$}}
\put(92,48){\makebox(0,0)[lc]{$\mathcal N_4$}}
\put(92,63){\makebox(0,0)[lc]{$\mathcal N_k$}}
\put(91,82){\makebox(0,0)[cc]{$\mathcal N_\omega$}}
\put(77,21){\makebox(0,0)[cc]{$\mathcal{RZ}$}}
\put(47,21){\makebox(0,0)[cc]{$\mathcal{SL}$}}
\put(47,0){\makebox(0,0)[cc]{$\mathcal T$}}
\put(92,18){\makebox(0,0)[lc]{$\mathcal{ZM}$}}
\end{picture}
\caption{Non-group and Abelian group chain varieties}
\label{chain varieties}
\end{center}
\end{figure}

Combining above observations, it is easy to verify the following

\begin{proposition}
\label{definable sets of atoms}
The varieties $\mathcal{SL}$ and $\mathcal{ZM}$, and the set of varieties
$$\{\mathcal A_p\mid p\ \text{is a prime number}\}$$
are definable. The varieties $\mathcal{LZ}$ and $\mathcal{RZ}$ are definable
up to duality.
\end{proposition}

\begin{proof}
By Lemma \ref{atoms}, all varieties mentioned in the proposition are atoms of
\textbf{SEM}. By Lemma \ref{neutral}, the varieties $\mathcal{SL}$ and
$\mathcal{ZM}$ are neutral elements in \textbf{SEM}, while $\mathcal{LZ}$,
$\mathcal{RZ}$ and $\mathcal A_p$ are not. Fig.\ \ref{chain varieties} shows
that the varieties $\mathcal{ZM}$ and $\mathcal A_p$ are proper subvarieties
of some chain varieties, while $\mathcal{SL}$, $\mathcal{LZ}$ and $\mathcal
{RZ}$ are not. Therefore the formulas
\begin{align*}
&\mathtt{SL}(x)\;\rightleftharpoons\;\mathtt A(x)\,\&\,\mathtt{Neut}(x)\,\&\,
(\forall y)\,\bigl(\mathtt{Ch}(y)\,\&\,x\le y\longrightarrow x=y
\bigr),\\
&\mathtt{ZM}(x)\;\rightleftharpoons\;\mathtt A(x)\,\&\,\mathtt{Neut}(x)\,\&\,
(\exists y)\,\bigl(\mathtt{Ch}(y)\,\&\,x<y\bigr)
\end{align*}
define the varieties $\mathcal{SL}$ and $\mathcal{ZM}$ respectively, while
the the formulas
\begin{align*}
&\mathtt{LZ\text{-}and\text{-}RZ}(x)\;\rightleftharpoons\;\mathtt A(x)\,\&\,
\neg\mathtt{Neut}(x)\,\&\,(\forall y)\,\bigl(\mathtt{Ch}(y)\,\&\,x\le y
\longrightarrow x=y\bigr),\\
&\mathtt{GrA}(x)\;\rightleftharpoons\;\mathtt A(x)\,\&\,\neg\mathtt{Neut}
(x)\,\&\,(\exists y)\,\bigl(\mathtt{Ch}(y)\,\&\,x<y\bigr)
\end{align*}
define the sets $\{\mathcal{LZ,RZ}\}$ and $\{\mathcal A_p\mid p$ is a prime
number$\}$ respectively.
\end{proof}

Note that in fact each of the group atoms $\mathcal A_p$ is individually
definable (see Theorem \ref{A_n} below). The definability of the varieties
$\mathcal{SL}$ and $\mathcal{ZM}$ and the definability up to duality of the
varieties $\mathcal{LZ}$ and $\mathcal{RZ}$ are mentioned in \cite
{Jezek-McKenzie-93} (see Theorem 1.11 and Lemma 4.3 there) without any
explicitly written formulas. Note also that the variety $\mathcal{ZM}$ can be
defined by several different ways. For the sake of completeness, one can
provide one more of them. It is a common knowledge that a semigroup variety
is completely regular [a nil-variety] if and only if it does not contain the
variety $\mathcal{ZM}$ [any atom of the lattice \textbf{SEM} except $\mathcal
{ZM}$]. The lattice of completely regular semigroup varieties is modular (see
\cite{Pastijn-90,Pastijn-91,Petrich-Reilly-90} or Section 6 in \cite
{Shevrin-Vernikov-Volkov-09}). In contrast, the lattice of nil-varieties does
not satisfy any non-trivial lattice identity \cite
{Burris-Nelson-71-infinite}. Combining these observations, we see that the
variety $\mathcal{ZM}$ can be defined by the formula
$$\mathtt{ZM'}(x)\;\rightleftharpoons\;\mathtt A(x)\,\&\,(\forall y,z,t)\,
\bigl(y,z,t\ngeq x\,\&\,z\le t\longrightarrow(y\vee z)\wedge t=(y\wedge t)
\vee z\bigr)\ldotp$$

Put $\mathcal{COM}=\var\,\{xy=yx\}$. As an immediate consequence of Lemma
\ref{definable sets of atoms}, we have the following

\begin{proposition}
\label{COM}
The variety $\mathcal{COM}$ is definable.
\end{proposition}

\begin{proof}
It is well known that the join of the varieties $\mathcal A_p$ where $p$ runs
over the set of all prime numbers coincides with the variety $\mathcal{COM}$
(see \cite{Evans-71}, for instance). Therefore the formula
$$\mathtt{COM}(x)\;\rightleftharpoons\;\min\nolimits_x\bigl\{(\forall y)\,
\bigl(\mathtt{GrA}(y)\longrightarrow y\le x\bigr)\bigr\}$$
defines the variety $\mathcal{COM}$.
\end{proof}

The following general fact will be used in what follows.

\begin{lemma}
\label{chain set}
If a countably infinite subset $S$ of a lattice $L$ is definable in $L$ and
forms a chain isomorphic to the chain of natural numbers under the order
relation in $L$ then every member of this set is definable in $L$.
\end{lemma}

\begin{proof}
Let $S=\{s_n\mid n\in\mathbb N\}$, $s_1<s_2<\cdots<s_n<\cdots$ and let $\Phi
(x)$ be the formula defining $S$ in $L$. We are going to prove the
definability of the element $s_n$ for each $n$ by induction on $n$. The
induction base is evident because the element $s_1$ is definable by the
formula $\min_x\bigl\{\Phi(x)\bigr\}$. Assume now that $n>1$ and the element
$s_{n-1}$ is definable by some formula $\Psi(x)$. Then the formula
$$\min\nolimits_x\bigl\{\Phi(x)\,\&\,(\exists y)\,\bigl(\Psi(y)\,\&\,y<x\bigr
)\bigr\}$$
defines the element $s_n$.
\end{proof}

The fact that the variety $\mathcal{ZM}$ is definable is a partial case of
the following

\begin{proposition}
\label{nil chain}
Every chain nil-variety of semigroups is definable.
\end{proposition}

\begin{proof}
All considerations here are based on Lemma \ref{chain var} and Fig.\ \ref
{chain varieties}. The variety $\mathcal N_\omega$ is defined by the formula
$$\mathtt{N_\omega}(x)\;\rightleftharpoons\;\max\nolimits_x\bigl\{\mathtt{Ch}
(x)\,\&\,(\exists y,z,t)\,\bigl(\mathtt{ZM}(y)\,\&\,y<z<t<x\bigr)\bigr\}
\ldotp$$
The formula
$$\mathtt{All\text{-}N_k}(x)\;\rightleftharpoons\;(\exists y)\,\bigl(\mathtt
{N_\omega}(y)\,\&\,x<y\bigr)$$
defines the set of varieties $\{\mathcal N_k\mid k\in\mathbb N\}$. Now Lemma
\ref{chain set} successfully applies with the conclusion that the variety
$\mathcal N_k$ is definable for each $k$. It remains to verify the
definability of the varieties $\mathcal N_3^c$ and $\mathcal N_3^2$. Both
these varieties (and only they) are chain varieties that contains $\mathcal
{ZM}$ and are not contained in $\mathcal N_\omega$; besides that the variety
$\mathcal N_3^c$ is commutative, while the variety $\mathcal N_3^c$ is not.
Therefore the formulas
\begin{align*}
&\mathtt{N_3^c}(x)\;\rightleftharpoons\;\mathtt{Ch}(x)\,\&\,(\exists y,z,
t)\,\bigl(\mathtt{ZM}(y)\,\&\,y\le x\,\&\,\mathtt{N_\omega}(z)\,\&\,x\nleq
z\,\&\,\mathtt{COM}(t)\,\&\,x\le t\bigr),\\
&\mathtt{N_3^2}(x)\;\rightleftharpoons\;\mathtt{Ch}(x)\,\&\,(\exists y,z,
t)\,\bigl(\mathtt{ZM}(y)\,\&\,y\le x\,\&\,\mathtt{N_\omega}(z)\,\&\,x\nleq
z\,\&\,\mathtt{COM}(t)\,\&\,x\nleq t\bigr)
\end{align*}
define the varieties $\mathcal N_3^c$ and $\mathcal N_3^2$ respectively.
\end{proof}

Note that every chain Abelian group variety $\mathcal A_{p^k}$ also is
definable (see Theorem \ref{A_n} below).

\section{Main sublattices of the lattice \textbf{SEM}}
\label{sublattices}

The lattice \textbf{SEM} contains a number of wide and important sublattices
(see Section 1 and Chapter 2 in \cite{Shevrin-Vernikov-Volkov-09}). In this
section we aim to show that many of these sublattices are definable (as sets
of varieties).

A semigroup variety is called \emph{overcommutative} if it contains the
variety $\mathcal{COM}$. It is a common knowledge that every semigroup
variety is either overcommutative or periodic. Thus the lattice \textbf{SEM}
is the disjoint union of two big sublattices: the lattice of all periodic
varieties and the lattice of all overcommutative varieties. One more
important sublattice of \textbf{SEM} is the lattice of all commutative
varieties. It is evident that the formulas
\begin{align*}
&\mathtt{Per}(x)\;\rightleftharpoons\;(\exists y)\,\bigl(\mathtt{COM}(y)\,
\&\,y\nleq x\bigr),\\
&\mathtt{OC}(x)\;\rightleftharpoons\;\bigl(\forall y)\,(\mathtt{COM}(y)
\longrightarrow y\le x\bigr),\\
&\mathtt{Com}(x)\;\rightleftharpoons\;(\forall y)\,\bigl(\mathtt{COM}(y)
\longrightarrow x\le y\bigr)
\end{align*}
define the sets of all periodic varieties, all overcommutative varieties and
all commutative varieties respectively. Thus we have the following

\begin{theorem}
\label{Per,OC,Com}
The sets of all periodic varieties, all overcommutative varieties and all
commutative varieties of semigroups are definable.\qed
\end{theorem}

As we have already mentioned in Section \ref{atoms and chain}, many important
sets of semigroup varieties admit a characterization in the language of atoms
of the lattice \textbf{SEM}. Several facts of such a type are summarised in
Table \ref{sets}; all these facts are verified in \cite{Aizenshtat-74}.

\begin{table}[tbh]
\begin{center}
\small
\begin{tabular}{|c|c|}
\hline
A semigroup variety is&if and only if it does not contain the varieties\\
\hline
\protect\rule{0pt}{10pt}a completely regular variety&$\mathcal{ZM}$\\
\hline
\protect\rule{0pt}{10pt}a completely simple variety&$\mathcal{SL,ZM}$\\
\hline
\protect\rule{0pt}{10pt}a periodic group variety&$\mathcal{LZ,RZ,SL,ZM}$\\
\hline
\protect\rule{0pt}{10pt}a combinatorial variety&$\mathcal A_p$ for all prime
$p$\\
\hline
\protect\rule{0pt}{10pt}a variety of idempotent semigroups&$\mathcal{ZM}$ and
$\mathcal A_p$ for all prime $p$\\
\hline
\protect\rule{0pt}{10pt}a nil-variety&$\mathcal{LZ,RZ,SL}$ and $\mathcal A_p$
for all prime $p$\\
\hline
\end{tabular}
\caption{A characterization of some sets of varieties\protect\rule{0pt}
{11pt}}
\label{sets}
\end{center}
\end{table}

Varieties of all types mentioned in Table \ref{sets} form sublattices in
\textbf{SEM}. Combining the facts from Table \ref{sets} with Proposition \ref
{definable sets of atoms}, we have the following

\begin{theorem}
\label{definable sets}
The sets of all completely regular varieties, all completely simple
varieties, all periodic group varieties, all combinatorial varieties, all
varieties of idempotent semigroups, and all nil-varieties of semigroups are
definable.\qed
\end{theorem}

Note that the definability of the set of all nil-varieties was mentioned in
\cite{Vernikov-07}. For convenience of references, we write in Table \ref
{formulas for sets} formulas defining the sets of varieties listed in Theorem
\ref{definable sets}.

\begin{table}[tbh]
\begin{center}
\small
\tabcolsep=3pt
\begin{tabular}{|c|c|}
\hline
The set of all&is defined by the formula\\
\hline
\protect\rule{0pt}{10pt}completely regular varieties&$\mathtt{CR}(x)\;
\rightleftharpoons\;(\forall y)\,\bigl(\mathtt A(y)\,\&\,y\le x
\longrightarrow\neg\mathtt{ZM}(y)\bigr)$\\
\hline
\protect\rule{0pt}{10pt}completely simple varieties&$\mathtt{CS}(x)\;
\rightleftharpoons\;(\forall y)\,\bigl(\mathtt A(y)\,\&\,y\le x
\longrightarrow\neg\mathtt{SL}(y)\,\&\,\neg\mathtt{ZM}(y)\bigr)$\\
\hline
\protect\rule{0pt}{10pt}periodic group varieties&$\mathtt{Gr}(x)\;
\rightleftharpoons\;(\forall y)\,\bigl(\mathtt A(y)\,\&\,y\le x
\longrightarrow\mathtt{GrA}(y)\bigr)$\\
\hline
\protect\rule{0pt}{10pt}combinatorial varieties&$\mathtt{Comb}(x)\;
\rightleftharpoons\;(\forall y)\,\bigl(\mathtt A(y)\,\&\,y\le x
\longrightarrow\neg\mathtt{GrA}(y)\bigr)$\\
\hline
\protect\rule{0pt}{10pt}varieties of idempotent&$\mathtt{Idemp}(x)\;
\rightleftharpoons\;(\forall y)\,\bigl(\mathtt A(y)\,\&\,y\le x
\longrightarrow\neg\mathtt{GrA}(y)\,\&\,\neg\mathtt{ZM}(y)\bigr)$\\
\protect\rule{0pt}{10pt}semigroups&\\
\hline
\protect\rule{0pt}{10pt}nil-varieties&$\mathtt{Nil}(x)\;\rightleftharpoons\;
(\forall y)\,\bigl(\mathtt A(y)\,\&\,y\le x\longrightarrow\mathtt{ZM}(y)\bigr
)$\\
\hline
\end{tabular}
\caption{Formulas defining some sets of varieties\protect\rule{0pt}{11pt}}
\label{formulas for sets}
\end{center}
\end{table}

One more interesting sublattice of the lattice \textbf{SEM} is the lattice of
all 0-reduced varieties.

\begin{theorem}
\label{0-red}
The set of all \textup0-reduced semigroup varieties is definable.
\end{theorem}

\begin{proof}
Put
$$\mathtt{LMod}(x)\;\rightleftharpoons\;(\forall y,z)\,\bigl(x\le y
\longrightarrow x\vee(y\wedge z)=y\wedge(x\vee z)\bigr)\ldotp$$
An element $x$ of a lattice $L$ such that the sentence $\mathtt{LMod}(x)$ is
true is called \emph{lower-modular}. It is verified in \cite{Vernikov-07}
that a semigroup variety is 0-reduced if and only if it is a nil-variety and
a lower-modular element of the lattice \textbf{SEM}. Therefore the formula
$$\mathtt{0\text{-}red}(x)\;\rightleftharpoons\;\mathtt{Nil}(x)\,\&\,\mathtt
{LMod}(x)$$
defines the set of all 0-reduced varieties.
\end{proof}

Theorem \ref{0-red} was verified in \cite{Jezek-McKenzie-93} (see Theorem
1.11 there) without an explicitly written formula defining the set of all
0-reduced varieties. The formula $\mathtt{0\text{-}red}(x)$ is given in \cite
{Vernikov-07}, while some slightly more complex formula defining the set of
all 0-reduced varieties was written earlier in \cite{Volkov-05}.

Note that in Sections \ref{fin deg} and \ref{permutative} we provide several
other definable sublattices of \textbf{SEM}.

\section{Varieties of finite degree}
\label{fin deg}

We call a semigroup variety $\mathcal V$ a \emph{variety of finite degree} if
all nil-semigroups in $\mathcal V$ are nilpotent; $\mathcal V$ is called a
variety of \emph{degree} $k$ if nilpotency degrees of nilsemigroups in
$\mathcal V$ are bounded by the number $k$ and $k$ is the least number with
this property. In this section we show that the set of all semigroup
varieties of finite degree and certain its important subsets and members are
definable.

\begin{theorem}
\label{fin degree}
The sets of all semigroup varieties of finite degree and of all semigroup
varieties of degree $k$ \textup(for an arbitrary natural number $k$\textup)
are definable.
\end{theorem}

\begin{proof}
According to Theorem 2 of \cite{Sapir-Sukhanov-81}, $\mathcal V$ is a variety
of finite degree if and only if $\mathcal{N_\omega\nsubseteq V}$. It is easy
to see also that $\mathcal V$ is a variety of degree $\le k$ if and only if
$\mathcal N_{k+1}\nsubseteq\mathcal V$, whence $\mathcal V$ is a variety of
degree $k$ if and only if $\mathcal N_{k+1}\nsubseteq\mathcal V$ but
$\mathcal N_k\subseteq\mathcal V$. Therefore the set of all varieties of
finite degree is definable by the formula
$$\mathtt{FinDeg}(x)\;\rightleftharpoons\;(\forall y)\,\bigl(\mathtt
{N_\omega}(y)\longrightarrow y\nleq x\bigr),$$
while the set of all varieties of degree $k$ is defined by the formila
$$\mathtt{Deg_k}(x)\;\rightleftharpoons\;(\forall y,z)\,\bigl(\mathtt N_k
(y)\,\&\,\mathtt N_{k+1}(z)\longrightarrow y\le x\,\&\,z\nleq x\bigr)$$
where $\mathtt N_m(x)$ (for any natural $m$) is the formula that defines the
variety $\mathcal N_m$.
\end{proof}

For a natural number $k$, we put $\mathcal{NILP}_k=\var\,\{x_1x_2\cdots x_k=
0\}$.

\begin{theorem}
\label{nilp}
The set of all nilpotent semigroup varieties is definable. For an arbitrary
natural number $k$, the variety $\mathcal{NILP}_k$ is definable.
\end{theorem}

\begin{proof}
Evidently, a semigroup variety is nilpotent if and only if it is a
nil-variety of finite degree. Therefore, the set of all nilpotent varieties
is definable by the formula
$$\mathtt{Nilp}(x)\;\rightleftharpoons\;\mathtt{Nil}(x)\,\&\,\mathtt{FinDeg}
(x)\ldotp$$
The variety $\mathcal{NILP}_k$ is the largest nil-variety of degree $k$,
whence the formula
$$\mathtt{NILP_k}(x)\;\rightleftharpoons\;\max\nolimits_x\bigl\{\mathtt{Nil}
(x)\,\&\,\mathtt{Deg_k}(x)\bigr\}$$
defines this variety.
\end{proof}

A natural subset of varieties of finite degree is formed by \emph{varieties
of semigroups with completely regular power}, that is varieties $\mathcal V$
with the following property: for any member $S\in\mathcal V$, there is a
natural number $n$ such that the semigroup $S^n$ is completely regular. A
variety $\mathcal V$ is called a variety \emph{with completely regular
$k^{\text{th}}$ power} if $S^k$ is completely regular for any $S\in\mathcal
V$ and $k$ is the least number with this property. To prove that the set of
all semigroup varieties with completely regular [$k^{\text{th}}$] power is
definable, we need some additional information.

Put $\mathcal P=\var\,\{xy=x^2y,\,x^2y^2=y^2x^2\}$. It is well known that the
variety $\mathcal P$ is generated by the 3-element semigroup
$$P=\{e,a,0\}=\langle a,e\mid e^2=e,\,ea=a,\,ae=0\rangle\ldotp$$
The semigroup $P$ and the variety $\mathcal P$ frequently appears in articles
devoted to different aspects of the theory of semigroup varieties.

\begin{lemma}
\label{cr power}
A semigroup variety $\mathcal V$ of finite degree \textup[of degree
$k$\textup] is a variety of semigroups with completely regular power \textup
[with completely regular $k^{\text{th}}$ power\textup] if and only if
$\mathcal P,\overleftarrow{\mathcal P}\nsubseteq\mathcal V$.\qed
\end{lemma}

This assertion is verified in \cite{Tishchenko-90} for varieties of
semigroups with completely regular power, and its variant for varieties of
semigroups with completely regular $k^{\text{th}}$ power can be verified
quite analogously.

\begin{proposition}
\label{P and P*}
The varieties $\mathcal P$ and $\overleftarrow{\mathcal P}$ are definable up
to duality.
\end{proposition}

\begin{proof}
It can be easily verified (and follows from results of \cite{Sukhanov-85},
for instance) that the varieties $\mathcal P$, $\overleftarrow{\mathcal P}$
and only they have the property that any proper subvariety of a variety is
contained in the variety $\mathcal{SL\vee ZM}$. Therefore the formula
$$\mathtt{P\text{-}and\text{-}}\!\overleftarrow{\mathtt P}(x)\;
\rightleftharpoons\;(\exists y,z)\,\bigl(\mathtt{SL}(y)\,\&\,\mathtt{ZM}(z)\,
\&\,(\forall t)\,(t<x\longrightarrow t\le y\vee z)\bigr)$$
defines the set $\{\mathcal P,\overleftarrow{\mathcal P}\}$.
\end{proof}

The statement that the set of all completely regular varieties is definable
(see Theorem \ref{definable sets}) is generalized by the following

\begin{theorem}
\label{cr power def}
The set of all varieties of semigroups with completely regular power is
definable. For every natural number $k$, the set of all varieties of
semigroups with $k^{\text{th}}$ completely regular power is definable.
\end{theorem}

\begin{proof}
Lemma \ref{cr power} immediately implies that the set of all varieties of
semigroups with completely regular power and the set of all varieties of
semigroups with $k^{\text{th}}$ completely regular power are defined by the
formulas
\begin{align*}
&\mathtt{CRPow}(x)\;\rightleftharpoons\;\mathtt{FinDeg}(x)\,\&\,(\forall y)\,
\bigl(\mathtt{P\text{-}and\text{-}}\!\overleftarrow{\mathtt P}(y)
\longrightarrow y\nleq x\bigr),\\
&\mathtt{CRPow_k}(x)\;\rightleftharpoons\;\mathtt{Deg_k}(x)\,\&\,(\forall
y)\,\bigl(\mathtt{P\text{-}and\text{-}}\!\overleftarrow{\mathtt P}(y)
\longrightarrow y\nleq x\bigr)
\end{align*}
respectively.
\end{proof}

\section{Commutative varieties}
\label{commut}

Here we are going to provide some series of definable varieties of
commutative semigroups. To achieve this aim, we need some auxiliary facts.
The following lemma follows from Lemma 2 of \cite{Volkov-89} and the proof of
Proposition 1 of the same article.

\begin{lemma}
\label{K+N}
If a periodic semigroup variety $\mathcal V$ does not contain the varieties
$\mathcal{LZ}$, $\mathcal{RZ}$, $\mathcal P$ and $\overleftarrow{\mathcal P}$
then $\mathcal{V=K\vee N}$ where $\mathcal K$ is a variety generated by a
monoid, while $\mathcal N$ is a nil-variety.\qed
\end{lemma}

Let $C_{m,1}$ denote the cyclic monoid $\langle a\mid a^m=a^{m+1}\rangle$ and
let $\mathcal C_m$ be the variety generated by $C_{m,1}$. It is clear that
$$\mathcal C_m=\var\,\{x^m=x^{m+1},\,xy=yx\}\ldotp$$
In particular, $C_{1,1}$ is the 2-element semilattice and $\mathcal C_1=
\mathcal{SL}$. For notation convenience we put also $\mathcal C_0=\mathcal
T$. The following lemma can be easily extracted from the results of \cite
{Head-68}.

\begin{lemma}
\label{comm monoid}
If a periodic semigroup variety $\mathcal V$ is generated by a commutative
monoid then $\mathcal{V=G\vee C}_m$ for some Abelian periodic group variety
$\mathcal G$ and some $m\ge0$.\qed
\end{lemma}

Lemmas \ref{K+N} and \ref{comm monoid} immediately imply

\begin{corollary}
\label{comb and commut}
If $\mathcal V$ is a commutative combinatorial semigroup variety then
$\mathcal{V=C}_m\vee\mathcal N$ for some $m\ge0$ and some nil-variety
$\mathcal N$.\qed
\end{corollary}

Let now $\mathcal V$ be a commutative semigroup variety with $\mathcal{V\ne
COM}$. Lemmas \ref{K+N} and \ref{comm monoid} imply that $\mathcal{V=G\vee
C}_m\vee\mathcal N$ for some Abelian periodic group variety $\mathcal G$,
some $m\ge0$ and some commutative nil-variety $\mathcal N$. Our aim in this
section is to provide formulas defining the varieties $\mathcal G$ and
$\mathcal C_m$.

As is well known, a periodic semigroup variety $\mathcal X$ contains the
greatest nil-subvariety. We denote this subvariety by $\Nil(\mathcal X)$.
Put
$$\mathcal D_m=\Nil(\mathcal C_m)=\var\,\{x^m=0,\,xy=yx\}$$
for every natural $m$. In particular, $\mathcal D_1=\mathcal T$ and $\mathcal
D_2=\mathcal N_\omega$. Now we are well prepared to verify

\begin{proposition}
\label{C_m}
For each $m\ge0$, the variety $\mathcal C_m$ is definable.
\end{proposition}

\begin{proof}
First, we are going to verify that the formula
$$\mathtt{All\text{-}C_m}(x)\;\rightleftharpoons\;\mathtt{Com}(x)\,\&\,
\mathtt{Comb}(x)\,\&\,(\forall y,z)\,\bigl(\mathtt{Nil}(y)\,\&\,x=y\vee z
\longrightarrow x=z\bigr)$$
defines the set of varieties $\{\mathcal C_m\mid m\ge0\}$. Let $\mathcal V$
be a semigroup variety such that the sentence $\mathtt{All\text{-}C_m}
(\mathcal V)$ is true. Then $\mathcal V$ is commutative and combinatorial.
Now Corollary \ref{comb and commut} successfully applies with the conclusion
that $\mathcal{M=C}_m\vee\mathcal N$ for some $m\ge0$ and some nil-variety
$\mathcal N$. The fact that the sentence $\mathtt{All\text{-}C_m}(\mathcal
V)$ is true shows that $\mathcal{M=C}_m$.

Let now $m\ge0$. We aim to verify that the sentence $\mathtt{All\text{-}C_m}
(\mathcal C_m)$ is true. It is evident that the variety $\mathcal C_m$ is
commutative and combinatorial. Suppose that $\mathcal C_m=\mathcal{M\vee N}$
where $\mathcal N$ is a nil-variety. It remains to check that $\mathcal{N
\subseteq M}$. We may assume without any loss that $\mathcal{N=\Nil(C}_m)=
\mathcal D_m$. It is clear that $\mathcal M$ is a commutative and
combinatorial variety. Corollary \ref{comb and commut} implies that
$\mathcal{M=C}_r\vee\mathcal N'$ for some $r\ge0$ and some nil-variety
$\mathcal N'$. Then $\mathcal{N'\subseteq\Nil(C}_m)=\mathcal N$, whence
$$\mathcal C_m=\mathcal{M\vee N=C}_r\vee\mathcal{N'\vee\mathcal N=C}_r\vee
\mathcal N\ldotp$$
It suffices to prove that $\mathcal{N\subseteq C}_r$ because $\mathcal{N
\subseteq C}_r\vee\mathcal{N'=M}$ in this case. The equality $\mathcal C_m=
\mathcal C_r\vee\mathcal N$ implies that $\mathcal C_r\subseteq\mathcal C_m$,
whence $r\le m$. If $r=m$ then $\mathcal{N\subseteq C}_r$, and we are done.
Let now $r<m$. Then the variety $\mathcal C_m=\mathcal C_r\vee\mathcal N$
satisfies the identity $x^ry^m=x^{r+1}y^m$. Recall that the variety $\mathcal
C_m$ is generated by a monoid. Substituting 1 for $y$ in this identity, we
obtain that $\mathcal C_m$ satisfies the identity $x^r=x^{r+1}$. Therefore
$\mathcal C_m\subseteq\mathcal C_r$ contradicting the unequality $r<m$.

Thus we have proved that the set of varieties $\{\mathcal C_m\mid m\ge0\}$ is
definable by the formula $\mathtt{All\text{-}C_m}(x)$. Now Lemma \ref
{chain set} successfully applies with the conclusion that the variety
$\mathcal C_m$ is definable for each $m$.
\end{proof}

The following assertion generalizes the fact that the variety $\mathcal
N_\omega$ is definable (see Proposition \ref{nil chain}).

\begin{proposition}
\label{x^m=0,xy=yx}
For every natural number $m$, the variety $\mathcal D_m$ is definable.
\end{proposition}

\begin{proof}
Put
$$\mathtt{Nil\text{-}part}(x,y)\;\rightleftharpoons\;\mathtt{Per}(x)\,\&\,y
\le x\,\&\,\mathtt{Nil}(y)\,\&\,(\forall z)\,\bigl(z\le x\,\&\,\mathtt{Nil}
(z)\longrightarrow z\le y\bigr)\ldotp$$
Clearly, for semigroup varieties $\mathcal X$ and $\mathcal Y$, the sentence
$\mathtt{Nil\text{-}part}(\mathcal{X,Y})$ is true if and only if $\mathcal X$
is periodic and $\mathcal{Y=\Nil(X)}$. Let $\mathtt{C_m}$ be the formula
defining the variety $\mathcal C_m$. The variety $\mathcal D_m$ is defined by
the formula
$$\mathtt{D_m}(x)\;\rightleftharpoons\;(\exists y)\,\bigl(\mathtt{C_m}(y)\,
\&\,\mathtt{Nil\text{-}part}(y,x)\bigr)$$
because $\mathcal D_m=\Nil(\mathcal C_m)$.
\end{proof}

To prove the definability of an arbitrary Abelian periodic group variety, we
need some definitions, notation and an auxiliary result. We denote by \textbf
{Com} the lattice of all commutative semigroup varieties. We call a
commutative semigroup variety 0-\emph{reduced in} \textbf{Com} if it may be
given by the commutative law and some non-empty set of 0-reduced identities
only. If $\mathcal X$ is a commutative nil-variety of semigroups then we
denote by $\ZR(\mathcal X)$ the least 0-reduced in \textbf{Com} variety that
contains $\mathcal X$. Clearly, the variety $\ZR(\mathcal X)$ is given by the
commutative law and all 0-reduced identities that hold in $\mathcal X$. If
$u$ is a word and $x$ is a letter then $c(u)$ denotes the set of all letters
occurring in $u$, while $\ell_x(u)$ stands for the number of occurrences of
$x$ in $u$.

\begin{lemma}
\label{A_n key}
Let $m$ and $n$ be natural numbers with $m>2$ and $n>1$. The following are
equivalent:
\begin{itemize}
\item[\textup{(i)}]$\Nil(\mathcal A_n\vee\mathcal{X)=\ZR(X)}$ for any variety
$\mathcal{X\subseteq D}_m$;
\item[\textup{(ii)}]$n\ge m-1$.
\end{itemize}
\end{lemma}

\begin{proof}
(i)$\longrightarrow$(ii) Suppose that $n<m-1$. Let $\mathcal X$ be the
subvariety of $\mathcal D_m$ given within $\mathcal D_m$ by the identity
\begin{equation}
\label{A_n key ident}
x^{n+1}y=xy^{n+1}\ldotp
\end{equation}
Since $n+1<m$, the variety $\mathcal X$ is not 0-reduced in \textbf{Com}. The
identity \eqref{A_n key ident} holds in the variety $\mathcal A_n\vee\mathcal
X$, and therefore in the variety $\Nil(\mathcal A_n\vee\mathcal X)$. But the
latter variety does not satisfy the identity $x^{n+1}y=0$ because this
identity fails in $\mathcal X$. We see that the variety $\Nil(\mathcal A_n
\vee\mathcal X)$ is not 0-reduced in \textbf{Com}. Since the variety $\ZR
(\mathcal X)$ is 0-reduced in \textbf{Com}, we are done.

(ii)$\longrightarrow$(i) Let $n\ge m-1$ and $\mathcal{X\subseteq D}_m$. One
can verify that $\mathcal A_n\vee\mathcal{X=A}_n\vee\ZR(\mathcal X)$. Note
that this equality immediately follows from Lemma 2.5 of \cite
{Shaprynskii-dnel} whenever $n\ge m$. We reproduce here the corresponding
arguments for the sake of completeness. It suffices to check that $\mathcal
A_n\vee\mathcal{\ZR(X)\subseteq A}_n\vee\mathcal X$ because the opposite
inclusion is evident. Suppose that the variety $\mathcal A_n\vee\mathcal X$
satisfies an identity $u=v$. We need to prove that this identity holds in
$\mathcal A_n\vee\ZR(\mathcal X)$. Since $u=v$ holds in $\mathcal A_n$, we
have $\ell_x(u)\equiv\ell_x(v)(\text{mod}\,n)$ for any letter $x$. If $\ell_x
(u)=\ell_x(v)$ for all letters $x$ then $u=v$ holds in $\mathcal A_n\vee\ZR
(\mathcal X)$ because this variety is commutative. Therefore we may assume
that $\ell_x(u)\ne\ell_x(v)$ for some letter $x$. Then either $\ell_x(u)\ge
n$ or $\ell_x(v)\ge n$. We may assume without any loss that $\ell_x(u)\ge n$.
Suppose that $n\ge m$. Then the identity $u=0$ holds in the variety $\mathcal
D_m$, whence it holds in $\mathcal X$. This implies that $v=0$ holds in
$\mathcal X$ too. Therefore the variety $\ZR(\mathcal X)$ satisfies the
identities $u=0=v$. Since the identity $u=v$ holds in $\mathcal A_n$, it
holds in $\mathcal A_n\vee\ZR(\mathcal X)$, and we are done.

It remains to consider the case $n=m-1$. If $\ell_x(u)\ge m$ or $\ell_x(v)\ge
m$ for some letter $x$, we go to the situation considered in the previous
paragraph. Thus, for any letter $x\in c(u)\cup c(v)$, either $\ell_x(u)=n-1$
and $x\notin c(v)$ or $x\notin c(u)$ and $\ell_x(v)=n-1$. We may assume
without any loss that the latter is the case. In particular $x\notin c(u)$.
Substituting 0 for $x$ in $u=v$, we obtain that the variety $\mathcal X$
satisfies the identity $u=0$. We go to the situation considered in the
previous paragraph again.

We have proved that $\mathcal A_n\vee\mathcal{X=A}_n\vee\ZR(\mathcal X)$.
Therefore $\ZR(\mathcal{X)\subseteq\Nil(A}_n\vee\mathcal X)$. If the variety
$\mathcal X$ satisfies an identity $u=0$ then $u^{n+1}=u$ holds in $\mathcal
A_n\vee\mathcal X$. This readily implies that $u=0$ in $\Nil(\mathcal A_n\vee
\mathcal X)$. Hence $\Nil(\mathcal A_n\vee\mathcal{X)\subseteq\ZR(X)}$. Thus
$\Nil(\mathcal A_n\vee\mathcal{X)=\ZR(X)}$.
\end{proof}

Now we are well prepared to prove the announced above

\begin{theorem}
\label{A_n}
An arbitrary Abelian periodic group variety is definable.
\end{theorem}

\begin{proof}
Abelian periodic group varieties are exhausted by the trivial variety and the
varieties $\mathcal A_n$ with $n>1$. The trivial variety is obviously
definable. For brevity, put
\begin{align*}
&\mathtt{Ab}(x)\;\rightleftharpoons\;\mathtt{Com}(x)\,\&\,\mathtt{Gr}(x),\\
&\mathtt{Com\text{-}0\text{-}red}(x)\;\rightleftharpoons\;(\exists y,z)\,
\bigl(\mathtt{COM}(y)\,\&\,\mathtt{0\text{-}red}(z)\,\&\,x=y\wedge z\bigr),\\
&\mathtt{ZR}(x,y)\;\rightleftharpoons\;\mathtt{Com\text{-}0\text{-}red}(y)\,
\&\,x\le y\,\&\,(\forall z)\,\bigl(\mathtt{Com\text{-}0\text{-}red}(z)\,\&\,x
\le z\longrightarrow y\le z\bigr)\ldotp
\end{align*}
The formula $\mathtt{Ab}(x)$ [respectively $\mathtt{Com\text{-}0\text{-}red}
(x)$] defines the set of all Abelian periodic group varieties [respectively
all 0-reduced in \textbf{Com} varieties] and, for semigroup varieties
$\mathcal X$ and $\mathcal Y$, the sentence $\mathtt{ZR}(\mathcal{X,Y})$ is
true if and only if $\mathcal{Y=\ZR(X)}$. Let $m$ be a natural number with
$m>2$. In view of Lemma \ref{A_n key}, the formula
$$\mathtt{A_{\ge m-1}}(x)\;\rightleftharpoons\;\mathtt{Ab}(x)\,\&\,(\forall
y,z,t)\,\bigl(\mathtt{D_m}(y)\,\&\,z\le y\,\&\,\mathtt{Nil\text{-}part}(x\vee
z,t)\longrightarrow\mathtt{ZR}(z,t)\bigr)$$
defines the set of varieties $\{\mathcal A_n\mid n\ge m-1\}$. Therefore the
formula
$$\mathtt{A_n}(x)\;\rightleftharpoons\;\mathtt{A_{\ge n}}(x)\,\&\,\neg\mathtt
{A_{\ge n+1}}(x)$$
defines the variety $\mathcal A_n$.
\end{proof}

\begin{corollary}
\label{comm monoid def}
A semigroup variety generated by a commutative monoid is definable.
\end{corollary}

\begin{proof}
Let $\mathcal V$ be a variety generated by some commutative monoid. According
to Lemma \ref{comm monoid}, $\mathcal{V=A}_n\vee\mathcal C_m$ for some $n\ge
1$ and $m\ge0$. It is easy to check that the parameters $n$ and $m$ in this
decomposition are defined uniquely. Therefore the formula
$$(\exists y,z)\,\bigl(\mathtt{A_n}(y)\,\&\,\mathtt{C_m}(z)\,\&\,x=y\vee z
\bigr)$$
defines the variety $\mathcal V$ (we assume here that $\mathtt{A_1}$ is the
evident formula defining the variety $\mathcal A_1=\mathcal T$).
\end{proof}

It was proved in \cite{Kisielewicz-04} that the set of all Abelian periodic
group varieties and each Abelian group variety are definable in the lattice
\textbf{Com}. Moreover, some characterization of all commutative semigroup
varieties definable in the lattice \textbf{Com} was found in \cite{Grech-09}.
Proposition \ref{COM} readily implies that a commutative semigroup variety is
definable in \textbf{SEM} whenever it is definable in \textbf{Com}. Thus
Theorem \ref{A_n} follows from results of \cite{Kisielewicz-04}. However the
articles \cite{Grech-09,Kisielewicz-04} contain no explicit first-order
formulas that define the set of all Abelian periodic group varieties or any
given Abelian periodic group variety or any other commutative variety in the
lattice \textbf{Com}.

\section{Finitely universal varieties}
\label{finitely universal}

Following \cite{Shevrin-Vernikov-Volkov-09}, we call a semigroup variety
\emph{finitely universal} if the subvariety lattice of this variety contains
an anti-isomorphic copy of the partition lattice over arbitrary finite
set. An interest to varieties with this property is motivated by the well
known fact that the subvariety lattice of a finitely universal variety does
not satisfy any non-trivial lattice identity.  It is known \cite
{Burris-Nelson-71-finite} that the variety $\mathcal{COM}$ is finitely
universal. Moreover, it is easy to see that $\mathcal{COM}$ is a minimal
finitely universal variety. Another known example of a minimal finitely
universal variety is the variety
$$\mathcal H=\var\,\{x^2=xyx=0\},$$
see \cite{Vernikov-Volkov-98}. The question whether or not minimal finitely
universal varieties differ from $\mathcal{COM}$ and $\mathcal H$ there exist
are unknown so far (see Section 12 of \cite{Shevrin-Vernikov-Volkov-09} for
more detailed comments). In this connection, it is interested to note that
both the varieties $\mathcal{COM}$ and $\mathcal H$ are definable.
The variety $\mathcal{COM}$ is definable by Proposition \ref{COM}. Here we
are going to check the definability of the variety $\mathcal H$. By the way,
we provide some other examples of definable 0-reduced varieties. Put
\begin{align*}
&\mathcal E_m=\var\,\{x^m=0\}\ (m\ \text{is a natural number}),\\
&\mathcal F=\var\,\{x^2y=xyx=yx^2=0\}\ldotp
\end{align*}
In particular, $\mathcal E_1=\mathcal T$.

\begin{proposition}
\label{E_m,F,H}
The varieties $\mathcal E_m$ \textup(for any natural number $m$\textup),
$\mathcal F$ and $\mathcal H$ are definable.
\end{proposition}

\begin{proof}
One can prove that the variety $\mathcal E_m$ is definable by the formula
$$\mathtt{E_m}(x)\;\rightleftharpoons\;\max\nolimits_x\bigl\{\mathtt{Nil}
(x)\,\&\,(\exists y,z)\,\bigl(\mathtt{COM}(y)\,\&\,\mathtt{D_m}(z)\,\&\,x
\wedge y=z\bigr)\bigr\}\ldotp$$
In other words, we are going to check that $\mathcal E_m$ is the greatest
nil-variety $\mathcal N$ with the property $\mathcal{N\wedge COM=D}_m$. The
equality $\mathcal E_m\wedge\mathcal{COM=D}_m$ is evident. Let $\mathcal N$
be a nil-variety with $\mathcal{N\wedge COM=D}_m$. Then the variety $\mathcal
{N \wedge COM}$ satisfies the identity $x^m=0$. Therefore one of the
varieties $\mathcal N$ and $\mathcal{COM}$ satisfies a non-trivial identity
of the form $x^m=u$. It is evident that $\mathcal{COM}$ does not satisfy any
non-trivial identity of such the form. Therefore the identity $x^m=u$ holds
in $\mathcal N$. It is easy to see that this identity implies $x^m=0$ in
arbitrary nil-variety. Thus $\mathcal N$ satisfies the identity $x^m=0$, that
is $\mathcal{N\subseteq E}_m$.

Put
$$\mathtt{Distr}(x)\;\rightleftharpoons\;(\forall y,z)\;\bigl(x\vee(y\wedge
z)=(x\vee y)\wedge(x\vee z)\bigr)\ldotp$$
An element $x$ of a lattice $L$ such that the sentence $\mathtt{Distr}(x)$ is
true is called \emph{distributive}. Distributive elements in the lattice
\textbf{SEM} are completely determined in \cite{Vernikov-Shaprynskii-10}. In
particular, it is proved there that a nil-variety is a distributive element
of \textbf{SEM} if and only if it is 0-reduced and satisfies the identities
$x^2y=xyx=yx^2=0$. Therefore the formula
$$\mathtt F(x)\;\rightleftharpoons\;\max\nolimits_x\bigl\{\mathtt{0\text{-}red}
(x)\,\&\,\mathtt{Distr}(x)\bigr\}$$
defines the variety $\mathcal F$. Since $\mathcal{H=E}_2\wedge\mathcal F$,
the formula
$$\mathtt H(x)\;\rightleftharpoons\;(\exists y,z)\,\bigl(\mathtt{E_2}(y)\,
\&\,\mathtt F(z)\,\&\,x=y\wedge z\bigr)$$
defines the variety $\mathcal H$.
\end{proof}

\section{Permutative varieties}
\label{permutative}

An identity of the form
\begin{equation}
\label{permut id}
x_1x_2\cdots x_n=x_{1\alpha}x_{2\alpha}\cdots x_{n\alpha}
\end{equation}
where $\alpha$ is a non-trivial permutation on the set $\{1,2,\dots,n\}$ is
called \emph{permutational}. The number $n$ is called a \emph{length} of this
identity. A semigroup variety is called \emph{permutative} if it satisfies
some permutational identity. Permutative varieties are natural and important
generalization of commutative ones. Here we are going to prove the
definability of the set of all permutative varieties and certain its
important members and subset.

\begin{theorem}
\label{permut}
The set of all permutative semigroup varieties is definable.
\end{theorem}

\begin{proof}
By Proposition 2 of \cite{Sapir-Sukhanov-81}, a semigroup variety is
permutative if and only if it does not contain all minimal non-Abelian
periodic group varieties, varieties of all completely simple semigroups over
Abelian groups of exponent $p$ for all prime $p$, the varieties
$$\mathcal{LRB}=\var\,\{x=x^2,\,xyx=xy\}\ \text{and}\ \mathcal{RRB}=\var\,
\{x=x^2,\,xyx=yx\},$$
and the variety $\mathcal H$. The set of all minimal non-Abelian periodic
group varieties is defined by the formula
$$\mathtt{JNAb}(x)\;\rightleftharpoons\;\min\nolimits_x\bigl\{\mathtt{Gr}
(x)\,\&\,\neg\mathtt{Com}(x)\bigr\}\ldotp$$
Put
$$\mathtt{Gr\text{-}part}(x,y)\;\rightleftharpoons\;\mathtt{Per}(x)\,\&\,y\le
x\,\&\,\mathtt{Gr}(y)\,\&\,(\forall z)\,\bigl(z\le x\,\&\,\mathtt{Gr}(z)
\longrightarrow z\le y\bigr)\ldotp$$
For semigroup varieties $\mathcal X$ and $\mathcal Y$, the sentence $\mathtt
{Gr\text{-}part}(\mathcal{X,Y})$ is true if and only if $\mathcal X$ is
periodic and $\mathcal Y$ is the greatest group subvariety of $\mathcal X$. A
variety of completely simple semigroups over groups of some prime exponent
$p$ is the largest completely simple variety $\mathcal V$ such that the
largest group subvariety of $\mathcal V$ is $\mathcal A_p$. Therefore the set
of all such varieties is defined by the formula
$$\mathtt{All\text{-}CSA_p}(x)\;\rightleftharpoons\;\max\nolimits_x\bigl
\{\mathtt{CS}(x)\,\&\,(\forall y)\,\bigl(\mathtt{Gr\text{-}part}(x,y)
\longrightarrow\mathtt{GrA}(y)\bigr)\bigr\}\ldotp$$
As is well known, the lattice of all varieties of idempotent semigroups has
the form shown in Fig.\ \ref{band varieties} where $\mathcal I=\var\{x=x^2\}$
(see \cite{Evans-71} or \cite{Shevrin-Vernikov-Volkov-09}, for instance). In
particular, we see that $\mathcal{LRB}$ [respectively $\mathcal{RRB}$] is the
largest variety of idempotent semigroups that does not contain the variety
$\mathcal{RZ}$ [respectively $\mathcal{LZ}$]. Therefore the set $\{\mathcal
{LRB,RRB}\}$ is defined by the formula
$$\mathtt{LRB\text{-}and\text{-}RRB}(x)\;\rightleftharpoons\;\max\nolimits_x
\bigl\{\mathtt{Idemp}(x)\,\&\,(\exists y)\,\bigl(\mathtt{LZ\text{-}and\text
{-}RZ}(y)\,\&\,y\nleq x\bigr)\bigr\}\ldotp$$
Combining the above observations we have that the formula
$$\mathtt{Perm}(x)\!\rightleftharpoons\!(\forall y)\bigl(y\le
x\!\longrightarrow\!\neg\mathtt{JNAb}(y)\&\neg\mathtt{All\text{-}CSA_p}(y)\&
\neg\mathtt{LRB\text{-}and\text{-}RRB}(y)\&\neg\mathtt H(y)\bigr)$$
defines the set of all permutative varieties.
\end{proof}

\begin{figure}[tbh]
\begin{center}
\unitlength=.8mm
\linethickness{.4pt}
\begin{picture}(42,88)
\put(1,15){\line(0,1){53}}
\put(1,15){\line(2,-1){20}}
\put(1,15){\line(2,1){20}}
\put(1,25){\line(2,-1){20}}
\put(1,25){\line(2,1){40}}
\put(1,45){\line(2,-1){40}}
\put(1,45){\line(2,1){40}}
\put(1,65){\line(2,-1){40}}
\put(1,65){\line(2,1){6}}
\put(21,5){\line(0,1){10}}
\put(21,5){\line(2,1){20}}
\put(21,15){\line(2,1){20}}
\put(21,25){\line(2,-1){20}}
\put(21,25){\line(0,1){10}}
\put(35,68){\line(2,-1){6}}
\put(41,15){\line(0,1){53}}
\dashline{0.65}(1,68)(1,78)
\dashline{0.65}(7,68)(17,73)
\dashline{0.65}(25,73)(35,68)
\dashline{0.65}(41,68)(41,78)
\put(1,15){\circle*{1.66}}
\put(1,25){\circle*{1.66}}
\put(1,35){\circle*{1.66}}
\put(1,45){\circle*{1.66}}
\put(1,55){\circle*{1.66}}
\put(1,65){\circle*{1.66}}
\put(21,5){\circle*{1.66}}
\put(21,15){\circle*{1.66}}
\put(21,25){\circle*{1.66}}
\put(21,35){\circle*{1.66}}
\put(21,55){\circle*{1.66}}
\put(21,85){\circle*{1.66}}
\put(41,15){\circle*{1.66}}
\put(41,25){\circle*{1.66}}
\put(41,35){\circle*{1.66}}
\put(41,45){\circle*{1.66}}
\put(41,55){\circle*{1.66}}
\put(41,65){\circle*{1.66}}
\put(21,88){\makebox(0,0)[cc]{$\mathcal I$}}
\put(0,35){\makebox(0,0)[rc]{$\mathcal{LRB}$}}
\put(0,15){\makebox(0,0)[rc]{$\mathcal{LZ}$}}
\put(42,35){\makebox(0,0)[lc]{$\mathcal{RRB}$}}
\put(42,15){\makebox(0,0)[lc]{$\mathcal{RZ}$}}
\put(21,12){\makebox(0,0)[cc]{$\mathcal{SL}$}}
\put(21,0){\makebox(0,0)[cc]{$\mathcal T$}}
\end{picture}
\caption{The lattice of varieties of idempotent semigroups}
\label{band varieties}
\end{center}
\end{figure}

For a natural number $n>1$, we denote by $\mathcal{PERM}_n$ the variety given
by all permutational identities of length $n$. In particular, $\mathcal
{PERM}_2=\mathcal{COM}$. The fact that the variety $\mathcal{COM}$ is
definable (see Proposition \ref{COM}) is a partial case of the following

\begin{proposition}
\label{permut-n}
For an arbitrary natural number $n$, the variety $\mathcal{PERM}_n$ is
definable.
\end{proposition}

\begin{proof}
It is easy to see that $\mathcal{PERM}_n=\mathcal{COM\vee NILP}_n$. Therefore
the formula
$$\mathtt{PERM_n}(x)\;\rightleftharpoons\;(\exists y,z)\,\bigl(\mathtt{COM}
(y)\,\&\,\mathtt{NILP_n}(z)\,\&\,x=y\vee z\bigr)$$
defines the variety $\mathcal{PERM}_n$.
\end{proof}

A semigroup variety is called \emph{strongly permutative} if it satisfies an
identity of the form \eqref{permut id} with $1\alpha\ne1$ and $n\alpha\ne n$.
It is proved in \cite{Putcha-Yaqub-71} that a variety $\mathcal V$ is
strongly permutative if and only if $\mathcal{V\subseteq PERM}_n$ for some
$n$.

\begin{theorem}
\label{strongly permut}
The set of all strongly permutative semigroup varieties is definable.
\end{theorem}

\begin{proof}
Let $\mathcal V$ be a strongly permutative variety. Then
$$\mathcal{V\subseteq PERM}_n=\mathcal{COM\vee NILP}_n$$
for some $n$. Thus $\mathcal V$ is contained in the join of the variety
$\mathcal{COM}$ and some nilpotent variety. Now, suppose that a variety
$\mathcal V$ is contained in the join of $\mathcal{COM}$ and some nilpotent
variety $\mathcal N$. Then $\mathcal{N\subseteq NILP}_n$ for some $n$.
Therefore $\mathcal{V\subseteq COM\vee NILP}_n=\mathcal{PERM}_n$, whence
$\mathcal V$ is strongly permutative. We have proved that a variety is
strongly permutative if and only if it is contained in the join of the
variety $\mathcal{COM}$ and some nilpotent variety. Therefore the formula
$$\mathtt{StrPerm}(x)\;\rightleftharpoons\;(\exists y,z)\,\bigl(\mathtt{COM}
(y)\,\&\,\mathtt{Nilp}(z)\,\&\,x\le y\vee z\bigr)$$
defines the set of all strongly permutative varieties.
\end{proof}

At the conclusion, we note that there are many other semigroup varieties
whose definability (or definability up to duality) may be confirmed by
explicitly written formulas. We mention only one remarkable class of
varieties of such a kind, namely the class of all varieties of idempotent
semigroups. Indeed, the variety $\mathcal I$ is defined by the formula
$\max\nolimits_x\bigl\{\mathtt{Idemp}(x)\bigr\}$. Formulas defining the
variety $\mathcal{SL}$ and defining up to duality the varities $\mathcal
{LZ}$, $\mathcal{RZ}$, $\mathcal{LRB}$ and $\mathcal{RRB}$ are given above.
Let now $\mathcal B$ be a variety of idempotent semigroups with $\mathcal{B
\ne I}$. Then the lattice $L(\mathcal B)$ is finite (see Fig.\ \ref
{band varieties}). Let $n$ be the length of this lattice. Basing on Fig.\
\ref{band varieties}, it is easy to write (by induction on $n$) the formula
that defines $\mathcal B$ whenever $\mathcal B=\overleftarrow{\mathcal B}$
and defines $\mathcal B$ up to duality whenever $\mathcal B\ne\overleftarrow
{\mathcal B}$.

\medskip

\textbf{Acknowledgement.} The author thanks Dr.\ Olga Sapir for many
stimulating discussions.

\end{document}